\documentclass[12pt,a4paper,oneside]{amsart}
\pdfoutput=1
\usepackage{amssymb}
\usepackage{esint}
\usepackage{mathtools}
\usepackage{hyperref}

\mathtoolsset{showonlyrefs}

\usepackage[english]{babel}
\usepackage[latin1]{inputenc}
\usepackage[pdftex]{graphicx}

\theoremstyle{plain}
\newtheorem{theorem}{Theorem}
\newtheorem{lemma}[theorem]{Lemma}
\newtheorem{corollary}[theorem]{Corollary}

\theoremstyle{definition}

\theoremstyle{remark}
\newtheorem{remark}[theorem]{Remark}

\newcommand{\T}{\mathbb{T}}
\newcommand{\R}{\mathbb{R}}
\newcommand{\Z}{\mathbb{Z}}
\newcommand{\C}{\mathbb{C}}
\newcommand{\N}{\mathbb{N}}

\newcommand{\p}{\mathcal{P}}
\newcommand{\pp}{\mathcal{Q}}
\newcommand{\tf}{\mathcal{T}}
\newcommand{\ip}[2]{\left\langle#1,#2\right\rangle}
\newcommand{\vev}[1]{\left\langle#1\right\rangle}
\newcommand{\der}{\mathrm{d}}

\renewcommand{\phi}{\varphi}
\newcommand{\abs}[1]{\left| #1 \right|}
\newcommand{\aabs}[1]{\left\| #1 \right\|}

\title{On Radon transforms on tori}
\author{Joonas Ilmavirta}
\address{Department of Mathematics and Statistics, University of Jyv\"askyl\"a, P.O.Box 35 (MaD) FI-40014 University of Jyv\"askyl\"a, Finland}
\email{joonas.ilmavirta@jyu.fi}
\date{\today}

\begin{document}

\begin{abstract}
We show injectivity of the X-ray transform and the $d$-plane Radon transform for distributions on the $n$-torus, lowering the regularity assumption in the recent work by Abouelaz and Rouvi\`ere.
We also show solenoidal injectivity of the X-ray transform on the $n$-torus for tensor fields of any order, allowing the tensors to have distribution valued coefficients.
These imply new injectivity results for the periodic broken ray transform on cubes of any dimension.
\end{abstract}



\maketitle

\section{Introduction}
\label{sec:intro}

The question we set out to study is whether a function on the standard flat torus $\T^n=\R^n/\Z^n$ is uniquely determined by its integrals over all periodic geodesics.
We also generalize this in two directions: we can either allow integrals over periodic $d$-planes for $d<n$ or allow the function to be a tensor field (which is to be determined up to gauge).
The answer to all of these questions is affirmative, and the functions involved can be any distributions.

The first question was studied by Abouelaz and Rouvi\`ere~\cite{AR:radon-torus}, and they showed uniqueness for functions on $\T^n$ that have Fourier transforms in $\ell^1(\Z^n)$, which includes $C^{n+1}(\T^n)$.
The same result on the two dimensional torus~$\T^2$ was given earlier by Strichartz~\cite{S:radon-variations}.
Later Abouelaz~\cite{A:torus-plane-radon} considered the same problem with $d$-planes and showed uniqueness under the same regularity assumption.
Their proof was based on constructing a normal operator, which caused the regularity assumption; our approach is different, and the assumption may be relaxed significantly.
A brief comparison of these methods is given in section~\ref{sec:methods}.

To the best of our knowledge, the result for tensor fields is new.
Injectivity of the X-ray transform for tensors of any order was shown by
Sharafutdinov~\cite{S:tensor-book} in the Euclidean space
and
Croke and Sharafutdinov~\cite{CS:negative-curvature-spectral} on closed manifolds with negative curvature.
There are also recent extensions to Anosov manifolds~\cite{PSU:anosov,PSU:anosov-mfld}.
These results are related to spectral geometry; this connection is elaborated on in section~\ref{sec:other}.
For the X-ray transform of tensor fields in~$\R^n$, see~\cite{S:tensor-book}; for the Radon transform of distributions on~$\R^n$, see~\cite{book-helgason}.

The rest of this introduction is organized as follows.
In section~\ref{sec:transform} we define the integral transforms and related basic concepts that we will use.
After that we are ready to state our results in section~\ref{sec:results}.
In sections~\ref{sec:methods} and~\ref{sec:other} we give an overview of the methods used and relations to other problems.
The necessary auxiliary results and definitions are given in section~\ref{sec:tools} and the results are finally proved in section~\ref{sec:proof}.

\subsection{Integral transforms}
\label{sec:transform}

The uniqueness problems studied here can be formulated in terms of injectivity of certain integral transforms.
We first define our transforms for functions in $\tf=C^\infty(\T^n)$; generalization by duality to distributions in~$\tf'$ is given in section~\ref{sec:distr}.
We parametrize periodic geodesics by $\pp=\Z^n\setminus\{0\}$ by letting $(x,v)\in\T^n\times\pp$ correspond to the periodic geodesic $[0,1]\ni t\mapsto x+vt$.
For any~$f\in\tf$, $x\in\T^n$ and~$v\in\pp$ we define
\begin{equation}
Rf(x,v)
=
\int_0^1f(x+vt)\der t.
\end{equation}
We refer to the function~$Rf$ defined on $\T^n\times\pp$ as the Radon transform of~$f$.
The function $Rf$ is somewhat redundantly defined, since $Rf(x+av,nv)=Rf(x,v)$ for any $a\in\R$ and $n\in\Z\setminus\{0\}$, but it does not matter.

We denote by~$G^n_d$ the set of linearly independent unordered $d$-tuples in~$\pp$; note that $G^n_1=\pp$.
We write $v\in A$ to denote that~$v$ is included in~$A$.
For any $f\in\tf$, $x\in\T^n$ and $A=\{v_1,\dots,v_d\}\in G^n_d$ we define
\begin{equation}
\begin{split}
R_df(x,A)
&=
R_df(x,v_1,\dots,v_d)
\\&=
\idotsint_{t\in[0,1]^d}f(x+t_1v_1+\dots+t_dv_d)\der t_1\dots\der t_d.
\end{split}
\end{equation}
We refer to the function~$R_df$ defined on $\T^n\times G^n_d$ as the $d$-plane Radon transform of~$f$.

It remains to define the Radon transform of a tensor function, to which end we first introduce some tensor notation.
Let $f=f_{i_1\cdots i_m}\der x^{i_1}\otimes\cdots\otimes\der x^{i_m}$ be a symmetric $m$-tensor field.
By symmetry we mean that the scalar functions $f_{i_1\cdots i_m}$ are invariant under permutations of indices.
For $(x,\xi)\in T\T^n$, we write
\begin{equation}
\label{eq:tensor-poly}
f(x,\xi)=f_{i_1\cdots i_m}(x)\xi^{i_1}\cdots \xi^{i_m}
\end{equation}
and the Radon transform of~$f$ is then
\begin{equation}
R^mf(x,v)=\int_0^1 f(x+tv,v)\der t,
\end{equation}
defined for all $x\in\T^n$ and $v\in\pp$.
By $f\in\tf$ we mean that each component of~$f$ is in~$\tf$, and similarly for other function spaces.
We remark that~$R^m$ may be applied to any $k$-tensor field, but its Radon transform only depends on its symmetric part.

Now the question whether a function or a tensor on the torus is uniquely determined by its integrals over all periodic geodesics or $d$-planes is equivalent with the one whether the integral transforms~$R$, $R_d$ and~$R^m$ are injective.
With obvious identifications we have $R=R_1=R^0$.
The transforms~$R^m$ are usually known as X-ray transforms, but for simplicity we refer to all of these integral transforms as Radon transforms.

The integral transform~$R^m$ for tensors of order $k\geq1$ is never injective.
For $k=1$ an easy calculation shows that the derivative~$\der f$ of a scalar function $f\in\tf$ has vanishing Radon transform, and analogously a tensor of any order has vanishing Radon transform if it is the derivative of a tensor of lower order (the definition is given in theorem~\ref{thm:tensor}).
The natural question is, whether this is the only obstruction to injectivity.
Injectivity modulo derivatives in this sense is known as solenoidal injectivity or $s$-injectivity (cf. eg.~\cite{PSU:anosov}).
The transforms~$R^m$ indeed turn out to be solenoidally injective.

\begin{remark}
It would also be possible to study the $d$-plane Radon transform~$R^m_d$ of a~$m$-tensor field, but we shall omit this.
It seems that one should assume that~$d$ divides~$m$ (so that $f(x,v)$ depends $m/d$-homogeneously on each component~$v^i$) and assume a suitable symmetry on the tensors.
\end{remark}

\subsection{Results}
\label{sec:results}

We make the standing assumption that $n\geq2$, $1\leq d<n$ and $m\geq1$.
As usual, we denote the set of all mappings $A\to B$ by~$B^A$ for any sets~$A$ and~$B$.

Above we defined the $d$-plane Radon transform of a function $f\in\tf$ as $R_df:\T^n\times G^n_d\to\R$.
As it will be more convenient, we shall think of the Radon transform~$R_df$ as a function on~$\T^n$ parametrized by~$G^n_d$.
That is, we consider the Radon transform $R_d$ as a mapping $\tf\to\tf^{G^n_d}$ by letting
\begin{equation}
R_df(A)(x)=R_df(x,A).
\end{equation}
Similar identification is made for the transform~$R^m$.

We establish the following results:

\begin{theorem}[Continuity and injectivity]
\label{thm:inj}
The $d$-plane Radon transform is an injective, linear, continuous map $\tf\to\tf^{G^n_d}$.
By continuity we mean that for any $A\in G^n_d$ the map $f\mapsto R_df(\cdot,A)$ is a continuous map $\tf\to\tf$.

It also extends to an injective, linear, continuous map $\tf'\to(\tf')^{G^n_d}$.
\end{theorem}

\begin{theorem}[Range characterization]
\label{thm:range}
Let $E\subset \tf'$ be any subspace.
We have
\begin{equation}
R_d(E)
=
\{F\in (\tf')^{G^n_d};\exists g\in E:\hat F(k,A)=\hat g(k)\prod_{v\in A}\delta_{0,v\cdot k}\}.
\end{equation}
\end{theorem}

\begin{theorem}[Pointwise recovery]
\label{thm:limit}
If $f\in C(\T^n)$, then we can construct a sequence of continuous functions from~$R_df$ that converges uniformly to~$f$.
In particular, we can reconstruct~$f$ pointwise.

If $f\in L^p(\T^n)$, $1\leq p<\infty$, then we can reconstruct a converging sequence in~$L^p$, which has a subsequence that converges to~$f$ almost everywhere.
\end{theorem}

The recovery result is based on reconstructing the Fourier coefficients of~$f$.
Indeed, for any $k\in\Z^n$ there exists $A\in G^n_d$ such that $\widehat{R_df}(k,A)=\hat f(k)$, where the Fourier transform is taken with respect to the variable on~$\T^n$.
This follows from the fact that for any $A\in G^n_d$ we have
\begin{equation}
\label{eq:radon-fourier}
\widehat{R_df}(k,A)=\hat f(k)\prod_{v\in A}\delta_{0,v\cdot k}.
\end{equation}
This result will be proven below.

We also have a stability result.
In suitable norms defined on the Fourier side the Radon transform is actually isometric.
We recall that the Sobolev space $H^s\subset\tf'$, $s\in\R$, is equipped with the norm
\begin{equation}
\aabs{f}_{H^s}
=
\sqrt{\sum_{k\in\Z^n}\vev{k}^{2s}\abs{\hat f(k)}^2},
\end{equation}
where $\vev{k}=\sqrt{1+\abs{k}^2}$.
The spaces~$H^s$ satisfy $\tf=H^\infty\coloneqq\bigcap_{s\in\R}H^s$ and $\tf'=H^{-\infty}\coloneqq\bigcup_{s\in\R}H^s$.
The corresponding norm~\eqref{eq:Triebel-norm} on $R_d(H^s)$ resembles that of the Triebel space $F^s_{2\infty}$.

\begin{theorem}[Stability]
\label{thm:stab}
For any $f\in\tf'$ we have
\begin{equation}
\abs{\hat f(k)}
=
\max_{A\in G^n_d}\abs{\widehat{R_df}(k,A)},
\end{equation}
where the Fourier transform $\widehat{R_df}$ is taken with respect to the first variable only.

In particular, for the Sobolev spaces~$H^s$, $s\in\R$, the $d$-plane Radon transform
\begin{equation}
R_d:H^s\to R_d(H^s)
\end{equation}
is an isometry when $R_d(H^s)$ is equipped with the norm
\begin{equation}
\label{eq:Triebel-norm}
\aabs{F}_{R_d(H^s)}
=
\sqrt{\sum_{k\in\Z^n}\vev{k}^{2s}\max_{A\in G^n_d}\abs{\hat F(k,A)}^2}.
\end{equation}
\end{theorem}

Theorem~\ref{thm:inj} implies an improvement for the recent injectivity result for the periodic broken ray transform in the cube~\cite[Proposition~30]{I:refl}.
For more details on the broken ray transform and its periodic version, see e.g.~\cite{I:refl} and section~\ref{sec:other} below.

\begin{corollary}
\label{cor:brt}
Let $Q=\prod_{i=1}^n[0,L_i]\subset\R^n$ be a rectangular box.
Let~$f$ be a sum of an~$L^1$ function and a compactly supported distribution in the interior of~$Q$.
Then the integral of~$f$ over all periodic billiard trajectories in~$Q$ determines~$f$ uniquely.
\end{corollary}

For tensor fields, the exact statement of solenoidal injectivity is as follows:

\begin{theorem}[Tensor fields]
\label{thm:tensor}
Let $f=f_{i_1\cdots i_m}\der x^{i_1}\otimes\cdots\otimes\der x^{i_m}$ be a symmetric tensor field of order~$m$ on~$\T^n$ with coefficients in~$\tf'$.
The Radon transform~$R^mf$ can be naturally defined by duality.
Then $R^mf=0$ if and only if there is a tensor field $h$ of order $m-1$ and regularity~$\tf'$ such that $f=\sigma\nabla h$, where~$\nabla$ is the Levi-Civita connection and~$\sigma$ is symmetrization.

Furthermore, if for some $s\in\R$ we have $f\in H^s(\T^n)$ and $R^mf=0$, then the tensor~$h$ can be chosen to be in $H^{s+1}(\T^n)$.
\end{theorem}

Corollary~\ref{cor:brt} can also be generalized to tensor fields as follows:

\begin{corollary}
\label{cor:brt-tensor}
Let $Q=\prod_{i=1}^n[0,L_i]\subset\R^n$ be a rectangular box.
Let~$f$ be a symmetric $m$-tensor field whose components are compactly supported distributions in the interior of~$Q$.
Then the integral of~$f$ over all periodic billiard trajectories in $Q$ determines $f$ uniquely.
\end{corollary}

\subsection{Outline of methods}
\label{sec:methods}

Abouelaz and Rouvi\`ere~\cite{AR:radon-torus} showed injectivity and pointwise recovery for the X-ray transform on the torus ($d=1$) for functions with Fourier series in $\ell^1(\Z^n)$ (in particular for $C^{n+1}(\T^n)$).
Their result extends that of Strichartz in dimension two~\cite{S:radon-variations}.
Later, Abouelaz~\cite{A:torus-plane-radon} extended the result for any $d<n$ with the same regularity assumptions.


Abouelaz and Rouvi\`ere use a normal operator.
There are only countably many $d$-planes passing through a given point on a torus, and this set is not equipped in an obvious way with any measure that could be used to construct the adjoint and thus the normal operator of~$R_d$; Abouelaz and Rouvi\`ere~\cite{AR:radon-torus} introduced an artificial weight in order to define the normal operator.
Because of this we do not think that there is a natural adjoint operator and find a different approach better suited for the problem at hand.
We do not use a normal operator of any kind.
We think of the Radon transform~$R_df$ of a function~$f$ as a function on~$\T^n$ parametrized by~$G^n_d$; in this interpretation it turns out that~$R_d$ is actually self adjoint, which makes further analysis easy and allows very low regularity.

Abouelaz and Rouvi\`ere parametrize geodesics and $d$-planes by the subset $\p\subset\pp$ that consists of vectors that are not integer multiples of vectors in~$\pp$ by other factors than $\pm1$.
(Equivalently, one may define that~$\p$ consists of those elements in~$\pp$ whose components have no common divisors other than~$\pm1$.)
This restriction loses no data, since $R_df(x,v)=R_df(x,nv)$ for any $n\in\Z\setminus\{0\}$.
Using all of~$\pp$ instead of~$\p$ makes Fourier analysis easier.


There are some earlier $s$-injectivity results for~$R^m$ in various spaces:
by Sharafutdinov in the Euclidean space~\cite{S:tensor-book}, by Croke and Sharafutdinov on manifolds with negative curvature~\cite{CS:negative-curvature-spectral}, and by Paternain, Salo and Uhlmann on Anosov manifolds with certain conditions~\cite{PSU:anosov,PSU:anosov-mfld}.
The approach used in~\cite{PSU:anosov,PSU:anosov-mfld} was based on fibrewise Fourier analysis and extending the transform $R^m$ to distributions on the sphere invariant under the geodesic flow.
The only appearance of fibrewise Fourier analysis here is in lemma~\ref{lma:Hs} below; the rest of Fourier analysis is global, which is possible since the torus allows for very simple and elegant global Fourier analysis.
Helgason~\cite{book-helgason} has shown that the Radon transform is injective on compatly supported distributions on~$\R^n$ using Fourier analysis.

%

%

\subsection{Relation to other problems}
\label{sec:other}

Injectivity of the X-ray transform on 2-tensors is related to spectral rigidity (and boundary rigidity for manifolds with boundary) at least in certain situations without conjugate points.
The connection between integral geometry and spectral geometry has been used, for example, in~\cite{GK:spectral-surface,GK:spectral-negative-curvature,CS:negative-curvature-spectral,PSU:anosov,PSU:anosov-mfld}.
For a survey of recent results in tensor tomography and its applications, see~\cite{PSU:tensor-survey}; for spectral geometry, see~\cite{DH:spectral-survey}.
We do not know whether the standard torus~$\T^n$ is spectrally rigid or determined by its spectrum.
There are, however, pairs of isospetral flat tori that are not isometric~\cite{M:16D-tori,CS:4D-tori}.
For more details on isospectral but not isometric manifolds, see~\cite{S:isospectral-different-geometry}.

Corollaries~\ref{cor:brt} and~\ref{cor:brt-tensor} give injectivity of the periodic broken ray transform in a rectangular domain.
This transform was introduced in~\cite{I:refl} as a generalization of the broken ray transform, where instead of periodic broken rays (billiard trajectories) one considers broken rays of finite length with both endpoints in a given subset of the boundary.

The broken ray transform arises, for example, in inverse boundary value problems with partial data.
Eskin~\cite{eskin} reduced recovery of the electromagnetic potential from partial Cauchy data for the Schr\"odinger equation to the injectivity of the broken ray transform.
Similarly, Kenig and Salo~\cite{KS:calderon} reduced Calder\'on's problem with partial data to injectivity of the broken ray transform.
Both results hold in a restricted geometric setting which is not important here; we only mention that the endpoints of the broken rays need to lie in the part of boundary available for measurements.

The broken ray transform and its periodic version have only recently begun to be studied.
Most results so far are for scalar fields -- the only result for tensor fields that we know of is by Eskin~\cite{eskin} who considered scalar and vector fields.
The reflection argument used here to reduce the periodic broken ray transform to a periodic X-ray transform has proved useful when the reflecting boundary is (piecewise) flat~\cite{I:refl,H:square,H:brt-flat-refl}.

\section{Tools}
\label{sec:tools}

This section contains auxiliary results and definition of Radon transforms on distributions.
The reader who is not interested in technical details may wish to skip to section~\ref{sec:proof}.

\subsection{The Radon transform on test functions and distributions}
\label{sec:distr}


We recall the definition that $\pp=\Z^n\setminus\{0\}$.
The subset $\p=\{v\in\pp;v=nu\text{ for no }n\in\Z\setminus\{-1,1\}\}\subset\pp$ is not necessary in our approach, but it was used in~\cite{A:torus-plane-radon,AR:radon-torus}, so it appears in lemmas borrowed from there.
We denote by~$G^n_d$ the set of unordered linearly independent $d$-tuples of vectors in~$\pp$.
The set~$G^n_d$ can be thought of as a parametrization of a discrete Grasmannian associated to the structure of the torus.


We remark that there is significant redundancy in parametrizing $d$-planes by~$G^n_d$; some of the redundancy can be removed by replacing~$\pp$ with~$\p$.
Two elements in~$G^n_d$ give rise to the same $d$-plane if they span the same subsets in~$\R^n$ (when one considers~$G^n_d$ to be a collection of vectors in $\pp\subset\R^n$).

We denote the duality pairing between $\tf'$ and $\tf$ by $\ip{\cdot}{\cdot}$.
An elementary calculation shows that for any $f,\eta\in\tf$ and $A\in G^n_d$ we have
\begin{equation}
\ip{f}{R_d\eta(\cdot,A)}
=
\ip{R_d f(\cdot,A)}{\eta}.
\end{equation}
Thus we may naturally define the transform~$R_d$ on~$\tf'$ by setting
\begin{equation}
\label{eq:distr-def}
(R_df(\cdot,A))(\eta)
=
\ip{f}{R_d\eta(\cdot,A)}
\end{equation}
for any $f\in\tf'$, $A\in G^n_d$ and $\eta\in\tf$.

It is evident that $f\mapsto R_df(\cdot,A)$ is continuous for any $A\in G^n_d$ as a mapping $\tf\to\tf$ and $\tf'\to\tf'$.
Equation~\eqref{eq:distr-def} is the unique continuous extension of~$R_d$ to distributions.

Let us recall the topologies on~$\tf$ and~$\tf'$.
The topology on~$\tf$ is induced by the seminorms
\begin{equation}
\tf\ni\phi\mapsto \sup_{\abs{\alpha}\leq N}\sup_{\T^n}\abs{\partial^\alpha\phi}
\end{equation}
indexed by~$N$, where~$N$ ranges through~$\N$.
As the family of seminorms is countable, the topology is metrizable.
On~$\tf'$ we use the weak star topology, where a sequence $(f_i)\subset\tf'$ converges to $f\in\tf'$ if and only if $\ip{f_i}{\phi}\to\ip{f}{\phi}$ for all~$\phi\in\tf$.

For any tensor function $f\in\tf$ and scalar $\eta\in\tf$ we have
\begin{equation}
\label{eq:distr-def-tensor}
\ip{R^mf(\cdot,v)}{\eta}
=
\ip{f(\cdot,v)}{R\eta(\cdot,v)}.
\end{equation}
This allows us to naturally define $R^mf(\cdot,v)\in\tf'$ whenever $f(\cdot,v)\in\tf'$ by demanding that~\eqref{eq:distr-def-tensor} holds also for $f\in\tf'$.
Again, $f\mapsto R^mf(\cdot,v)$ is continuous as a mapping $\tf\to\tf$ and $\tf'\to\tf'$ for any $v\in\pp$.
The Fourier transform $\widehat{R^mf}$ is taken with respect to~$\T^n$ only.

\subsection{Auxiliary results}
\label{sec:aux}

The following result is stated as a part of a proof in \cite[p.~11]{A:torus-plane-radon}.
In the notation therein, the lemma states that $\psi(k)>0$ for all $k\in\Z^n$.

\begin{lemma}
\label{lma:A-perp}
For any $k\in\Z^n$ there is $A\in G^n_d$ such that $k\cdot v=0$ for all~$v\in A$.
\end{lemma}

The proof of this lemma is elementary, but we omit it here.

Our results will rely on Fourier analysis.
To that end, we define the function $e_k\in\tf$ for any $k\in\Z^n$ by setting $e_k(x)=\exp(2\pi ik\cdot x)$.
The Fourier components of a distribution $f\in\tf'$ are given by
\begin{equation}
\hat f(k)
=
\ip{f}{e_{-k}}.
\end{equation}
For a function on $\T^n\times G^n_d$ (such as the $d$-plane transform of a function) the Fourier transform is taken with respect to~$\T^n$ only.
That is, by~\eqref{eq:distr-def}
\begin{equation}
\label{eq:R-F-def}
\widehat{R_df}(k,A)
=
\ip{f}{R_de_{-k}(\cdot,A)}
\end{equation}
for any $f\in\tf'$, $k\in\Z^n$ and $A\in G^n_d$.

By equation~\eqref{eq:R-F-def} it is evident that we need to know the transforms~$R_de_k$.
These are very simple, as the following lemma demonstrates.

\begin{lemma}
\label{lma:Re}
For $k\in\Z^n$ we have
\begin{equation}
R_de_k(x,A)
=
e_k(x)
\prod_{v\in A}\delta_{0,v\cdot k}.
\end{equation}
\end{lemma}

\begin{proof}
Let $A=\{v_1,\dots,v_d\}\subset\pp$.
Then
\begin{equation}
\begin{split}
R_de_k(x,A)
&=
\idotsint_{t\in[0,1]^d}e_k(x+t_1v_1+\dots+t_dv_d)\der t_1\dots\der t_d
\\&=
e_k(x)\prod_{i=1}^d\int_0^1e_k(tv_i)\der t,
\end{split}
\end{equation}
from which the claim follows.
\end{proof}

For theorem~\ref{thm:tensor} we also need a simple result from Fourier analysis.

\begin{lemma}
\label{lma:fourier}
%
Suppose $F(x,v)$ is a polynomial in $v\in\Z^n$ and a distribution in $x\in\T^n$.
That is, $F(x,v)=\sum_{\abs{\alpha}\leq m}F_\alpha(x)v^\alpha$ for some natural number~$m$, where $F_\alpha\in\tf'$ for each~$\alpha$.
Suppose $\hat F(k,v)=0$ whenever $v\cdot k=0$, where the Fourier transform is with respect to the first variable.
Then there is a function $G(x,v)$ which is also a polynomial in $v\in\Z^n$ and a distribution in $x\in\T^n$ such that $\hat F(k,v)=v\cdot k\hat G(k,v)$ and the order of the polynomial $G(k,\cdot)$ is one lower than that of $F(k,\cdot)$.

Furthermore, if $F(\cdot,v)\in H^s$, then $G(\cdot,v)\in H^s$, and if $F(k,v)$ is homogeneous in~$v$, so is~$G(k,v)$.
\end{lemma}

\begin{proof}
We assume that~$F$ is homogeneous in~$v$ of degree~$m$; the proof of the general case is similar.
We also assume that $F(\cdot,v)\in H^s$.
Since $\bigcup_{s\in\R}H^s=\tf'$, this leads to no loss of generality.




Fix any $k\neq0$.
Now $\hat F(k,v)$ is a homogeneous polynomial in~$v\in\Z^n$, and we can uniquely extend it to a homogeneous polynomial in~$v\in\R^n$.
We decompose~$v\in\R^n$ in two components $v_\parallel\in\R$ and $v_\perp\in\R^n$ that satisfy $v=k\abs{k}^{-1}v_\parallel+v_\perp$ and $v_\perp\perp k$.
To fix the sign, we let $v_\parallel=\abs{k}^{-1}k\cdot v\in\R$.
For some coefficients $a_{k,m,\alpha}\in\C$ we have
\begin{equation}
\label{eq:vv1}
\hat F(k,v)
\eqqcolon
\hat F(k,v_\parallel,v_\perp)
=
\sum_{\abs{\alpha}\leq m}a_{k,m,\alpha}v_\perp^\alpha v_\parallel^{m-\abs{\alpha}},
\end{equation}
where~$\alpha$ is a multi-index.
But we know that $\hat F(k,0,v_\perp)=0$, so
\begin{equation}
\label{eq:vv2}
\sum_{\abs{\alpha}=m}a_{k,m,\alpha}v_\perp^\alpha=0.
\end{equation}
Combining equations~\eqref{eq:vv1} and~\eqref{eq:vv2} gives
\begin{equation}
\begin{split}
\hat F(k,v_\parallel,v_\perp)
&=
\sum_{\abs{\alpha}< m}a_{k,m,\alpha}v_\perp^\alpha v_\parallel^{m-\abs{\alpha}}
\\&=
v_\parallel\sum_{\abs{\alpha}\leq m-1}a_{k,m,\alpha}v_\perp^\alpha v_\parallel^{m-1-\abs{\alpha}}.
\end{split}
\end{equation}
We now set
\begin{equation}
\Gamma(k,v_\parallel,v_\perp)
\coloneqq
\abs{k}^{-1}\sum_{\abs{\alpha}\leq m-1}a_{k,m,\alpha}v_\perp^\alpha v_\parallel^{m-1-\abs{\alpha}}.
\end{equation}
Then, after reverting to the variable $v=k\abs{k}^{-1}v_\parallel+v_\perp$, we have $\hat F(k,v)=v\cdot k\Gamma(k,v)$.

For $k=0$ we let $\Gamma(0,v)=0$.
By construction it is clear that $\Gamma(k,v)$ is a homogeneous polynomial in~$v$ of degree~$m-1$.
We define the function~$G$ by letting $\Gamma(k,v)=\hat G(k,v)$.
It satisfies $\hat F(k,v)=v\cdot k\hat G(k,v)$.

It remains to prove the claim that $G(\cdot,v)\in H^s$ for every $v\in\Z^n$.
Since $F(\cdot,v)\in H^s$, we have $\sum_{k\in\Z^n}\abs{\hat{F}(k,v)}^2\vev{k}^{2s}<\infty$ for every $v\in\Z^n$.
But now $\abs{v\cdot k}^{-1}\leq1$ on the support of $\hat F$, so also $\sum_{k\in\Z^n}\abs{\hat G(k,v)}^2\vev{k}^{2s}<\infty$ for every $v\in\Z^n$.
\end{proof}

\begin{lemma}
\label{lma:Hs}
In the notation of theorem~\ref{thm:tensor}, if $h\in H^s$ is a symmetric tensor field on~$\T^n$ and also $\sigma\nabla h\in H^s$, then~$h\in H^{s+1}$.
\end{lemma}

\begin{proof}
The proof is based on fibrewise spherical harmonics and a decomposition of the geodesic vector field.
For details of these see~\cite{DS:killing-tensor,GK:spectral-negative-curvature,PSU:anosov-mfld}.

We denote $f=\sigma\nabla h$. 
As symmetric tensor fields~$h$ and~$f$ can be thought of as homogeneous polynomials on (each fibre of)~$T\T^n$ via~\eqref{eq:tensor-poly} and can be restricted to functions on the sphere bundle $S\T^n=\{(x,v)\in T\T^n;\abs{v}=1\}$.
If the degree of~$h$ is $m+1$, then $h=h_0+\cdots+h_{m-1}$ and $f=f_0+\cdots+f_m$, where~$h_i$ and~$f_i$ are spherical harmonics of order~$i$.
By assumption $h,f\in H^s$, that is, $h_i,f_i\in H^s$ for all~$i$.

Functions on~$S\T^n$ with fixed degree~$l$ (in terms of spherical harmonics) can be identified with symmetric trace free tensor fields of order~$l$ on~$\T^n$.
But the tensor~$h$ need not be trace free, corresponding to the fact that we may have $h_i\neq0$ for $i<m-1$ and similarly for~$f$.

If~$X$ denotes the geodesic vector field on $S\T^n$, we have after the aforementioned identification that $f=Xh$.
The geodesic vector field decomposes as $X=X_++X_-$, where $X_\pm$ changes the order of the spherical harmonic by~$\pm1$.
Applying this decomposition of $X$ to the equation $f=Xh$ with~$f$ and~$h$ written in terms of spherical harmonics and collecting terms of order~$m$, we have $X_+h_{m-1}=f_m$.
Since $f_m\in H^s$ and $X_+$ is overdetermined elliptic, we have $h_{m-1}\in H^{s+1}$.
Similarly the equation for terms of order $m-1$ gives $h_{m-2}\in H^{s+1}$.
For order $m-2$ we have $X_+h_{m-3}+X_-h_{m-1}=f_{m-2}$, so $X_+h_{m-3}=f_{m-2}-X_-h_{m-1}\in H^s$ and thus $h_{m-3}\in H^{s+1}$.
Carrying on inductively, we find indeed $h\in H^{s+1}$.
\end{proof}

\section{Proofs of theorems}
\label{sec:proof}

With the tools given above, the proofs of our results are elementary.
The proof of corollary~\ref{cor:brt} is essentially the same as given in~\cite[Proposition~30]{I:refl}.

\subsection{Proof of theorem~\ref{thm:inj}}

As the Fourier transform is injective on~$\tf'$, it suffices to reconstruct the Fourier coefficients of $f\in\tf'$ from~$R_df$.
To that end, let $k\in\Z^n$.
Combining equation~\eqref{eq:R-F-def} and lemma~\ref{lma:Re} proves equation~\eqref{eq:radon-fourier}.
By lemma~\ref{lma:A-perp} there is $A\in G^n_d$ such that $k\cdot v=0$ for all $v\in A$, so
\begin{equation}
\label{eq:R-F-thm}
\widehat{R_df}(k,A)=\hat f(k).
\end{equation}
Thus we may indeed construct~$\hat f$ from $\widehat{R_df}$.

\subsection{Proof of theorem~\ref{thm:range}}

The theorem follows from equation~\eqref{eq:radon-fourier}.
If~$g\in E$, then by equation~\eqref{eq:radon-fourier} the function~$R_dg$ has the required form.
If $F\in(\tf')^{G^n_d}$ is of the form given in the theorem, then~$F=R_dg$ again by~\eqref{eq:radon-fourier}.

\subsection{Proof of theorem~\ref{thm:limit}}


By theorem~\ref{thm:inj} we know the Fourier coefficients of~$f$ from~$R_df$.
If $f\in C(\T^n)$, convolutions of $f$ with the Fej\'er kernel converge uniformly to~$f$ and are smooth.

If $f\in L^p(\T^n)$, then using the same approximate identity given by the Fej\'er kernel, we obtain a convergent sequence\footnote{The Fourier series converges to~$f$ if $1<p<\infty$. For $f\in L^1$ this is not always the case (see e.g.~\cite[p.~50, exercise~3]{book-katznelson}).} in~$L^p$.

\subsection{Proof of theorem~\ref{thm:stab}}

The theorem follows immediately from equations~\eqref{eq:radon-fourier} and~\eqref{eq:R-F-thm}.

\subsection{Proof of corollary~\ref{cor:brt}}

Dilating the box $[0,L_1]\times\cdots\times[0,L_n]$ by any factor in a coordinate direction preserves all periodic broken rays.
Thus by scaling in each direction we may assume that$L_1=\dots=L_n=1/2$.

Let~$f$ be a sum of a compactly supported distribution in $(0,1/2)^n$ and an~$L^1$ function on $Q=[0,1/2]^n$.
We define $\tilde f$ on $[-1/2,1/2]^n$ by letting
\begin{equation}
\tilde{f}
=
f\circ\zeta,
\end{equation}
where
\begin{equation}
\zeta(x_1,\dots,x_n)
=
(\abs{x_1},\dots,\abs{x_n}).
\end{equation}
We can naturally identify $[-1/2,1/2]^n$ with~$\T^n$, and we have thus $\tilde f\in\tf(\T^n)$.

Let us assume for the moment that $f\in C(Q)$ and thus $\tilde f\in C(\T^n)$.
Let~$\gamma$ be any periodic geodesic in~$\T^n$.
Then $\zeta\circ\gamma$ is a periodic billiard trajectory on~$Q$, and the integral of~$f$ vanishes over it by assumption.
That is, $\tilde f\circ\gamma=f\circ(\zeta\circ\gamma)$ has zero integral and this holds for all~$\gamma$, so $R_1\tilde f=0$.
Thus, by theorem~\ref{thm:inj} $\tilde f=0$ and so also~$f=0$.

The definition of the integral of~$f$ over a periodic billiard trajectory can be naturally generalized to the case of a sum of a compactly supported distribution and an~$L^1$ function by transforming the situation to a torus and using the generalization given in section~\ref{sec:distr}.

\subsection{Proof of theorem~\ref{thm:tensor}}

Let $f\in\tf'$ be a tensor field which satisfies $R^mf=0$ in the sense of~\eqref{eq:distr-def-tensor}.
Analogously to~\eqref{eq:radon-fourier}, we have
\begin{equation}
\widehat{R^mf}(k,v)
=
\ip{f(\cdot,v)}{Re_{-k}(\cdot,v)}
=
\delta_{0,v\cdot k}\ip{f(\cdot,v)}{e_{-k}}
=
\delta_{0,v\cdot k}\hat f(k,v)
\end{equation}
for all $v\in\pp$ and $k\in\Z^n$.

Using the assumption $R^mf=0$ and lemma~\ref{lma:fourier}, we have thus $\hat{f}(k,v)=k\cdot v \hat g(k,v)$ for some function~$g$ for which $g(\cdot,v)\in\tf'$ for all $v\in\pp$ and~$g$ is a homogeneous polynomial of degree $m-1$ in~$v$.
Denoting $D=v\cdot\nabla_x$, we have that $\widehat{Dg}(k,v)=ik\cdot v\hat g(x,v)$, so $f(x,v)=-iDg(x,v)$.

Letting $h=-ig$, we have $f=Dh$ and $h\in\tf'$.
Because of its homogeneity (as a polynomial in~$v$) the function~$h$ can be considered a tensor field of order $m-1$ on~$\T^n$.
Since~$f$ is symmetric, the equation $f=Dh$ remain true if we replace~$h$ with its symmetrization and symmetrize~$Dh$.

By the last part of lemma~\ref{lma:fourier}, $f\in H^s$ implies that we can choose $h\in H^s$.
But since~$h$ is symmetric, we have in fact $h\in H^{s+1}$ by lemma~\ref{lma:Hs}.

\subsection{Proof of corollary~\ref{cor:brt-tensor}}

The proof is completely analogous to that of corollary~\ref{cor:brt} and is left to the reader.

\section*{Acknowledgements}
The author was partly supported by the Academy of Finland (Centre of Excellence in Inverse Problems Research).
The author wishes to thank Mikko Salo for help with lemma~\ref{lma:Hs} and for making several useful remarks and the referees for comments.

\bibliographystyle{abbrv}
\bibliography{disk}

\end{document}